%% file: BMR14_close_seb.tex
\documentclass[11pt]{article}
\usepackage{amsmath}
\usepackage{amssymb,amsbsy}
\usepackage[dvips]{graphicx}
\usepackage{pstricks}
\usepackage{pst-eps}
\usepackage{pst-node}
\usepackage{enumerate}
\usepackage{dsfont}
\usepackage{natbib}
\usepackage[margin=1in]{geometry}
\usepackage{tikz} \usetikzlibrary{arrows,decorations.pathmorphing,backgrounds,positioning,fit,petri,mindmap,shapes.geometric,decorations.pathreplacing}

\usepackage{pdfsync}
\usepackage{hyperref}

\allowdisplaybreaks[1]

\input{Commands}

\newcommand{\PA}{\mathrm{PA}}

\newcommand{\TV}{\mathrm{TV}}
\newcommand{\Beta}{\mathrm{Beta}}
\newcommand{\GGa}{\mathrm{GGa}}
\newcommand{\Dir}{\mathrm{Dir}}

\newcommand{\BlackBox}{\rule{1.5ex}{1.5ex}}  % end of proof
\newenvironment{proof}{\par\noindent{\bf Proof\ }}{\hfill\BlackBox\\[2mm]}

\DeclareMathOperator{\erf}{erf}
\DeclareMathOperator{\erfc}{erfc}

\begin{document}
\title{On the influence of the seed graph in \\ the preferential attachment model}
\author{
	S{\'e}bastien Bubeck
	\thanks{Princeton University; \texttt{sbubeck@princeton.edu}.}
	\and
	Elchanan Mossel
	\thanks{University of California, Berkeley; \texttt{mossel@stat.berkeley.edu}.}
	\and
	Mikl\'os Z.\ R\'acz
	\thanks{University of California, Berkeley; \texttt{racz@stat.berkeley.edu}.}
}
\date{\today}

\maketitle

\begin{abstract}
We study the influence of the seed graph in the preferential attachment model, focusing on the case of trees. 
We first show that the seed has no effect from a weak local limit point of view. 
On the other hand, we conjecture that different seeds lead to different distributions of limiting trees from a total variation point of view. 
We take a first step in proving this conjecture by showing that seeds with different degree profiles lead to different limiting distributions for the (appropriately normalized) maximum degree, implying that such seeds lead to different (in total variation) limiting trees.
\end{abstract}

\section{Introduction} \label{sec:intro}
We are interested in the following question: suppose we generate a large graph according to the linear preferential attachment model---can we say anything about the initial (seed) graph? A precise answer to this question could lead to new insights for the diverse applications of the preferential attachment model. In this paper we initiate the theoretical study of the seed's influence. 
Experimental evidence of the seed's influence already exists in the literature, see, e.g.,~\cite{schweiger2011generative}.
For sake of simplicity we focus on \emph{trees} grown according to linear preferential attachment.

\medskip

For a tree $T$ denote by $d_T(u)$ the degree of vertex $u$ in $T$, $\Delta(T)$ the maximum degree in $T$, and $\vec{d}(T) \in \N^{\N}$ the vector of degrees arranged by decreasing order.\footnote{We artificially continue the vector of degrees with zeros after the $|T|^{\text{th}}$ coordinate to put all degree profiles on the same space.} We refer to $\vec{d}(T)$ as the degree profile of $T$. For $n \geq k \geq 2$ and a tree $T$ on $k$ vertices we define the random tree $\PA(n, T)$ by induction. First $\PA(k, T)=T$. Then, given  $\PA(n,T)$, $\PA(n+1,T)$ is formed from  $\PA(n,T)$ by adding a new vertex $u$ and a new edge $uv$ where $v$ is selected at random among vertices in $\PA(n,T)$ according to the following probability distribution:
$$\P\left(v = i \ \middle| \, \PA(n, T) \right) = \frac{d_{\PA(n, T)}(i)}{2 \left( n - 1 \right)}.$$
This model was introduced in \cite{Mah92} under the name {\em Random Plane-Oriented Recursive Trees} but we use here the modern terminology of Preferential Attachment graphs, see \cite{BA99, BRST01}. 
In the following we also denote by $S_k$ the $k$-vertex star.

\medskip

We want to understand whether there is a relation between $T$ and $\PA(n, T)$ when $n$ becomes very large. We investigate three ways to make this question more formal. They correspond to three different points of view on the limiting tree obtained by letting $n$ go to infinity. 

\medskip

The least refined point of view is to consider the tree $\PA(\infty, T)$ defined on a countable set of vertices that one obtains by continuing the preferential attachment process indefinitely. As observed in \cite{KK05}, in this case the seed does not have any influence: indeed for any tree $T$, almost surely, $\PA(\infty, T)$ will be the unique isomorphism type of tree with countably many vertices and in which each vertex has infinite degree.  In fact this statement holds for {\em any} model where the degree of each fixed vertex diverges to infinity as the tree grows. For example, this notion of limit does not allow to distinguish between linear and non-linear preferential attachment models (as long as the degree of each fixed node diverges to infinity).

\medskip

Next we consider the much more subtle and fine-grained notion of a weak local limit introduced in \cite{BS01}. 
The notion of graph limits is more powerful than the one considered in the previous paragraph as it can, for example, distinguish between models having different limiting degree distributions. 
The weak local limit of the preferential attachment graph was first studied in the case of trees in~\cite{rudas2007random} using branching process techniques, and then later in general in \cite{BBCS12} using P\'olya urn representations. 
These papers show that $\PA(n, S_2)$ tends to the so-called P{\'o}lya-point graph in the weak local limit sense, and our first theorem utilizes this result to obtain the same for an arbitrary seed:
\begin{theorem} \label{th:weaklimit}
For any tree $T$ the weak local limit of $\PA(n, T)$ is the P{\'o}lya-point graph described in \cite{BBCS12} with $m=1$.
\end{theorem}
This result says that ``locally'' (in the Benjamini-Schramm sense) the seed has no effect. 
The intuitive reason for this result is that in the preferential attachment model most nodes are far from the seed graph and therefore it is expected that their neighborhoods will not reveal any information about it.

\medskip

Finally, we consider the most refined point of view, which we believe to be the most natural one for this problem as well as the richest one (both mathematically and in terms of insights for potential applications). 
First we rephrase our main question in the terminology of hypothesis testing. 
Given two potential seed trees $T$ and $S$, and an observation $R$ which is a tree on $n$ vertices, one wishes to test whether $R \sim \PA(n, T)$ or $R \sim \PA(n, S)$. Our original question then boils down to whether one can design a test with asymptotically (in $n$) non-negligible power. This is equivalent to studying the total variation distance between $\PA(n, T)$ and $\PA(n, S)$. Thus we naturally define 
$$\delta(S, T) = \lim_{n \to \infty} \mathrm{TV}(\PA(n, S), \PA(n, T)),$$
where $\TV$ denotes the total variation distance.\footnote{Observe that $\mathrm{TV}(\PA(n, S), \PA(n, T))$ is non-increasing in $n$ (since one can simulate the future evolution of the process) and always nonnegative so the limit is well-defined.} One can propose a test with asymptotically non-negligible power (i.e., a non-trivial test) iff $\delta(S,T) >0$. 
We believe that in fact this is always the case (except in trivial situations); precisely we make the following conjecture:

\begin{conjecture}\label{conj:main}
 $\delta$ is a metric on isomorphism types of trees with at least 3 vertices.\footnote{Clearly $\delta$ is a pseudometric on isomorphism types of trees with at least $3$ vertices so the only non-trivial part of the statement is that $\delta(S,T) \neq 0$ for $S$ and $T$ non-isomorphic.}
\end{conjecture}

While we have not yet been able to prove this conjecture,  we are able to distinguish trees with different degree profiles.

\begin{theorem}\label{th:diff_size_deg}
Let $S$ and $T$ be two finite trees on at least $3$ vertices. If $\vec{d}(S) \neq \vec{d}(T)$, then $\delta \left( S, T \right) > 0$.
\end{theorem}

In fact our proof shows a stronger statement, namely that different degree profiles lead to different limiting distributions for the (appropriately normalized) maximum degree.

\medskip

The smallest pair of trees we cannot as of yet distinguish is depicted in Figure~\ref{fig:example}. 
%\psset{unit=1pt}
%\psset{linewidth=0.8pt}
%\begin{figure}[ht]
%  \centering
%    \begin{pspicture*}(300,35)
%      % tree on the left
%      \psline{-}(10,5)(110,5)
%      \psline{-}(35,5)(35,30)
%      \qdisk(10,5){3}
%      \qdisk(35,5){3}
%      \qdisk(60,5){3}
%      \qdisk(85,5){3}
%      \qdisk(110,5){3}
%      \qdisk(35,30){3}
%      \rput{0}(100,25){$S$}
%
%      % tree on the right
%      \psline{-}(190,5)(290,5)
%      \psline{-}(240,5)(240,30)
%      \qdisk(190,5){3}
%      \qdisk(215,5){3}
%      \qdisk(240,5){3}
%      \qdisk(265,5){3}
%      \qdisk(290,5){3}
%      \qdisk(240,30){3}
%      \rput{0}(280,25){$T$}
%
%    \end{pspicture*}
%  \caption{Two trees with six vertices. Is $\delta \left( S , T \right) > 0$?}\label{fig:example}
%\end{figure}

\begin{figure}
\begin{center}
\begin{tikzpicture}[scale=1]
\draw (0,0) -- (1,0) -- (2,0) -- (3,0) -- (4,0);
\draw (1,0) -- (1,1);
\fill (0,0) circle (0.1);
\fill (1,0) circle (0.1);
\fill (2,0) circle (0.1);
\fill (3,0) circle (0.1);
\fill (4,0) circle (0.1);
\fill (1,1) circle (0.1);
\node at (3.5,1) {$S$};
\node at (9.5,1) {$T$};
\draw (6,0) -- (7,0) -- (8,0) -- (9,0) -- (10,0);
\draw (8,0) -- (8,1);
\fill (6,0) circle (0.1);
\fill (7,0) circle (0.1);
\fill (8,0) circle (0.1);
\fill (9,0) circle (0.1);
\fill (10,0) circle (0.1);
\fill (8,1) circle (0.1);
\end{tikzpicture}
\end{center}
\caption{Two trees with six vertices and $\vec{d}(S) = \vec{d}(T)$. Is $\delta \left( S , T \right) > 0$?}\label{fig:example}
\end{figure}
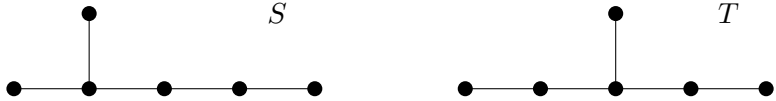

\medskip

In some cases we can say more. For instance, the distance between a fixed tree and a star can be arbitrarily close to $1$ if the star is large enough.

\begin{theorem}\label{th:arbstars}
For any fixed tree $T$ one has
\[
\lim_{k \to \infty} \delta \left( S_{k}, T \right) = 1.
\]
\end{theorem}

In the next section we derive results on the limiting distribution of the maximum degree $\Delta(\PA(n,T))$ that are useful in proving Theorems~\ref{th:diff_size_deg} and~\ref{th:arbstars}, which we then prove in Section~\ref{sec:proofs}. 
In Section~\ref{sec:candidate} we describe a particular way of generalizing the notion of maximum degree which we believe should provide a way to prove Conjecture~\ref{conj:main}. At present we are missing a technical result which we state separately as Conjecture~\ref{conj:technical} in the same section. 
The proof of Theorem~\ref{th:weaklimit} is in Section~\ref{sec:weaklimit}, while the proof of a key lemma described in Section~\ref{sec:useful} is presented in Section~\ref{sec:lemma_proof}. We conclude the paper with open problems in Section~\ref{sec:op}.

\section{Useful results on the maximum degree} \label{sec:useful}

We first recall several results that describe the limiting degree distributions of preferential attachment graphs (Section~\ref{sec:previous}), and from these we determine the tail behavior of the maximum degree in Section~\ref{sec:tail}, which we then use in the proofs of Theorems~\ref{th:diff_size_deg} and~\ref{th:arbstars}. 
Throughout the paper we label the vertices of $\PA(n,T)$ by $\left\{1,2,\dots,n\right\}$ in the order in which they are added to the graph, with the vertices of the initial tree labeled in decreasing order of degree, i.e., satisfying $d_T (1) \geq d_T (2) \geq \dots \geq d_T \left(\left| T \right| \right)$ (with ties broken arbitrarily). 
We also define the constant
\begin{equation}\label{eq:const}
 c\left(a,b\right) = \frac{\Gamma \left(2a-2\right)}{2^{b-1} \Gamma \left( a - 1/2\right) \Gamma \left( b \right)},
\end{equation}
which will occur multiple times.

\subsection{Previous results}\label{sec:previous}

\subsubsection{Starting from an edge}\label{sec:previous_edge}

\cite{mori2005maximum} used martingale techniques to study the maximum degree of the preferential attachment tree starting from an edge, and showed that $\Delta(\PA(n, S_2))/\sqrt{n}$ converges almost surely to a random variable which we denote by $D_{\max}\left(S_2\right)$. 
He also showed that for each fixed $i \geq 1$, $d_{\PA(n, S_2)}\left(i\right) / \sqrt{n}$ converges almost surely to a random variable which we denote by $D_i\left(S_2\right)$, and furthermore that $D_{\max} \left(S_2\right)= \max_{i \geq 1} D_i \left(S_2\right)$ almost surely. In light of this, in order to understand $D_{\max} \left(S_2 \right)$ it is useful to study $\left\{ D_i \left(S_2\right) \right\}_{i \geq 1}$. 
\cite{mori2005maximum} computes the joint moments of $\left\{ D_i \left(S_2\right)\right\}_{i \geq 1}$; in particular, we have (see~\cite[eq.~(2.4)]{mori2005maximum}) that for $i \geq 2$,
\begin{equation}\label{eq:mom_Janson}
 \E D_i \left( S_2 \right)^r = \frac{\Gamma \left( i - 1 \right) \Gamma \left( 1 + r \right)}{\Gamma \left( i - 1 + \frac{r}{2} \right)}.
\end{equation}

\medskip

Using different methods and slightly different normalization,~\cite{pekoz2013degree} also study the limiting distribution of $d_{\PA(n, S_2)}\left(i\right)$; in particular, they give an explicit expression for the limiting density. Fix $s \geq 1/2$ and define
\[
 \kappa_s \left( x \right) = \Gamma \left( s \right) \sqrt{\frac{2}{s \pi}} \exp \left( - \frac{x^2}{2s} \right) U \left( s - 1, \frac{1}{2}, \frac{x^2}{2s} \right) \mathbf{1}_{\left\{ x > 0 \right\}},
\]
where $U\left( a, b, z \right)$ denotes the confluent hypergeometric function of the second kind, also known as the Kummer U function (see~\cite[Chapter~13]{abramowitz1964handbook}); it can be shown that this is a density function. 
\cite{pekoz2013degree} show that for $i \geq 2$ the distributional limit of 
\[
 d_{\PA(n, S_2)}\left(i\right) / \left( \E d_{\PA(n, S_2)}\left(i\right)^2 \right)^{1/2}
\]
 has density $\kappa_{i-1}$ (they also give rates of convergence to this limit in the Kolmogorov metric).
Let $W_s$ denote a random variable with density $\kappa_s$. The moments of $W_s$ (see~\cite[Section~2]{pekoz2013degree}) are given by 
\begin{equation}\label{eq:mom_Pekoz}
 \E W_s^r = \left( \frac{s}{2} \right)^{r/2} \frac{\Gamma \left( s \right) \Gamma \left( 1 + r \right)}{\Gamma \left( s + \frac{r}{2} \right)},
\end{equation}
and thus comparing~\eqref{eq:mom_Janson} and~\eqref{eq:mom_Pekoz} we see that $D_i \left(S_2 \right) \stackrel{d}{=} \sqrt{2/\left(i-1\right)} W_{i-1}$ for $i \geq 2$.

\subsubsection{Starting from an arbitrary seed graph}\label{sec:previous_arbitrary}

Since we are interested in the effect of the seed graph, we desire similar results for $\PA(n,T)$ for an arbitrary tree $T$. 
One way of viewing $\PA(n,T)$ is to start growing a preferential attachment tree from a single edge and condition on it being $T$ after reaching $\left|T\right|$ vertices; $\PA(n,T)$ has the same distribution as $\PA(n,S_2)$ conditioned on $\PA\left(\left|T\right|,S_2\right) = T$. 
Due to this the almost sure convergence results of~\cite{mori2005maximum} carry over to the setting of an arbitrary seed tree. 
Thus for every fixed $i \geq 1$, $d_{\PA\left(n,T\right)} \left( i \right) / \sqrt{n}$ converges almost surely to a random variable which we denote by $D_i \left(T\right)$, $\Delta \left( \PA \left( n, T \right) \right) / \sqrt{n}$  converges almost surely to a random variable which we denote by $D_{\max} \left( T \right)$, and furthermore $D_{\max} \left( T \right) = \max_{i \geq 1} D_i \left( T \right)$ almost surely.

\medskip

In order to understand these limiting distributions, the basic observation is that for any $i$, $1 \leq i \leq \left|T\right|$, $\left( 2 \left( n - 1 \right) - d_{\PA\left(n,T\right)} \left( i \right), d_{\PA\left(n,T\right)} \left( i \right) \right)$ evolves according to a P\'olya urn with replacement matrix $\left(\begin{smallmatrix} 2 & 0 \\ 1 & 1 \end{smallmatrix}\right)$ starting from $\left(2\left(\left|T\right|-1\right) - d_T \left( i \right), d_T \left( i \right) \right)$. 
Indeed, when a new vertex is added to the tree, either it attaches to vertex $i$, with probability $d_{\PA\left(n,T\right)} \left( i \right) / \left( 2n - 2 \right)$, in which case both $d_{\PA\left(n,T\right)} \left( i \right)$ and $2 \left(n-1\right) - d_{\PA\left(n,T\right)} \left( i \right)$ increase by one, or otherwise it attaches to some other vertex in which case $d_{\PA\left(n,T\right)} \left( i \right)$ does not increase but $2\left(n-1\right) - d_{\PA\left(n,T\right)} \left( i \right)$ increases by two. 
\cite{Jan06} gives limit theorems for triangular P\'olya urns, and also provides information about the limiting distributions; for instance~\cite[Theorem~1.7]{Jan06} gives a formula for the moments of $D_i \left( T\right)$, extending~\eqref{eq:mom_Janson} for arbitrary trees $T$: for every $i$, $1 \leq i \leq \left|T\right|$, we have
\begin{equation}\label{eq:mom_Janson_gen}
 \E D_i \left( T \right)^r = \frac{\Gamma \left( \left|T\right| - 1 \right) \Gamma \left( d_T \left( i \right) + r \right)}{\Gamma \left( d_T \left( i \right) \right) \Gamma \left( \left|T\right| - 1 + \frac{r}{2} \right)},
\end{equation}
and for $i > \left|T\right|$ we have $\E D_i \left( T \right)^r = \Gamma \left( i - 1 \right) \Gamma \left( 1 + r \right) / \Gamma \left( i - 1 + r/2 \right)$, just like in~\eqref{eq:mom_Janson}.

\medskip 

The joint distribution of the limiting degrees in the seed graph, $\left( D_1 \left( T \right), \dots, D_{\left|T \right|} \left( T \right) \right)$, can be understood by viewing the evolution of $\left(d_{\PA\left(n,T\right)} \left( 1 \right), \dots, d_{\PA \left(n,T\right)} \left( \left|T\right| \right) \right)$ in the following way. When adding a new vertex, first decide whether it attaches to one of the initial $\left|T\right|$ vertices (with probability $\sum_{i=1}^{\left|T\right|} d_{\PA \left(n,T\right)} \left( i \right) / \left( 2n - 2 \right)$) or not (with the remaining probability); if it does, then independently pick one of them to attach to with probability proportional to their degrees. In other words, if viewed at times when a new vertex attaches to one of the initial $\left|T\right|$ vertices, the joint degree counts of the initial vertices evolve like a standard P\'olya urn with $\left|T\right|$ colors and identity replacement matrix.

\medskip

Let $\Beta \left(a,b\right)$ denote the beta distribution with parameters $a$ and $b$ (with density proportional to $x^{a-1} (1-x)^{b-1} \mathbf{1}_{\left\{ x \in [0,1]\right\}}$), let $\Dir\left( \alpha_1, \dots \alpha_s \right)$ denote the Dirichlet distribution with density proportional to $x_1^{\alpha_1-1} \cdots x_s^{\alpha_s-1} \mathbf{1}_{\left\{ x \in \left[0,1\right]^s, \sum_{i=1}^s x_i = 1 \right\}}$, and write  $X \sim \GGa \left(a,b\right)$ for a random variable $X$ having the generalized gamma distribution with density proportional to $x^{a-1} e^{-x^b} \mathbf{1}_{\left\{x > 0 \right\}}$. 
On the one hand,  $\left( 2 \left( n - 1 \right) - \sum_{i=1}^{\left|T\right|} d_{\PA \left(n,T\right)} \left( i \right), \sum_{i=1}^{\left|T\right|} d_{\PA \left(n,T\right)} \left( i \right) \right)$ evolves according to a P\'olya urn with replacement matrix $\left(\begin{smallmatrix} 2 & 0 \\ 1 & 1 \end{smallmatrix}\right)$ starting from $\left(0, 2(\left|T\right|-1)\right)$. \cite{Jan06} gives the limiting distribution of $\sum_{i=1}^{\left|T\right|} d_{\PA \left(n,T\right)} \left( i \right) / \sqrt{n}$ (see Theorem~1.8 and Example~3.1): $\sum_{i=1}^{\left|T\right|} D_i \left( T \right) \stackrel{d}{=} 2Z_{\left|T\right|}$, where $Z_{\left|T\right|} \sim \GGa \left( 2\left|T\right| - 1, 2 \right)$. On the other hand, it is known that in a standard P\'olya urn with identity replacement matrix the vector of proportions of each color converges almost surely to a random variable with a Dirichlet distribution with parameters given by the initial counts. These facts, together with the observation in the 
previous 
paragraph, lead to the following representation: if $X$ and $Z_{\left|T\right|}$ are independent, $X \sim \Dir \left(d_T \left( 1 \right), \dots, d_T \left( \left| T \right| \right) \right)$, and $Z_{\left|T\right|} \sim \GGa \left( 2 \left| T \right| - 1, 2 \right)$, then
\begin{equation}\label{eq:rep}
 \left( D_1 \left( T \right), \dots, D_{\left| T \right|} \left( T \right) \right) \stackrel{d}{=} 2 Z_{\left|T\right|} X.
\end{equation}
Recently,~\cite{pekoz2014joint} gave useful representations for $\left(D_1 \left( T \right), \dots, D_r \left( T \right) \right)$ for general $r$, and the representation above appears as a special case (see~\cite[Remark~1.9]{pekoz2014joint}).

\subsection{Tail behavior}\label{sec:tail}

In order to prove Theorem~\ref{th:diff_size_deg} our main tool is to study the tail of the limiting degree distributions. In particular, we use the following key lemma.

\begin{lemma}\label{lem:that_we_need}
Let $T$ be a finite tree.
 \begin{enumerate}[(a)]
  \item Let $U \subseteq \left\{1,2,\dots,\left|T\right| \right\}$ be a nonempty subset of the vertices of $T$, and let $d = \sum_{i \in U} d_T \left( i \right)$. Then 
      \begin{equation}\label{eq:tail_seed}
       \P \left( \sum_{i \in U} D_i \left( T \right) > t \right) \sim c \left( \left| T \right|, d \right) t^{1 - 2\left|T\right| + 2d} \exp \left( - t^2 / 4 \right)
      \end{equation}
  as $t \to \infty$, where the constant $c$ is as in~\eqref{eq:const}.\footnote{Throughout the paper we use standard asymptotic notation; for instance, $f\left( t \right) \sim g\left( t \right)$ as $t\to\infty$ if $\lim_{t\to\infty} f\left( t \right) / g \left( t \right) = 1$.}  
  \item For every $L > \left| T \right|$ there exists a constant $C \left( L \right) < \infty$ such that for every $t \geq 1$ we have
      \begin{equation}\label{eq:tail_late}
	\sum_{i=L}^\infty \P \left( D_i \left( T \right) > t \right) \leq C \left( L \right) t^{3-2L} \exp \left( - t^2 / 4 \right).
      \end{equation}
 \end{enumerate}
\end{lemma}

We postpone the proof of Lemma~\ref{lem:that_we_need} to Section~\ref{sec:lemma_proof}, as it results from a lengthy computation.
As an immediate corollary we get the asymptotic tail behavior of $D_{\max} \left( T \right)$.

\begin{corollary}\label{cor:max_tail}
Let $T$ be a finite tree and let $m := \left| \left\{ i \in \left\{ 1, \dots, \left|T\right| \right\} : d_T \left( i \right) = \Delta \left( T \right) \right\} \right|$. Then
\begin{equation}\label{eq:max_tail_exact}
 \P \left( D_{\max} \left( T \right) > t \right) \sim m \times c \left( \left| T \right|, \Delta \left( T \right) \right) t^{1 - 2\left|T\right| + 2 \Delta \left( T \right)} \exp \left( - t^2 / 4 \right)
\end{equation}
as $t \to \infty$, where the constant $c$ is as in~\eqref{eq:const}.
\end{corollary}
\begin{proof}
Recall the fact that $D_{\max} \left( T \right) = \max_{i \geq 1} D_i \left( T \right)$ almost surely. First, a union bound gives us that
\[
 \P \left( D_{\max} \left( T \right) > t \right) \leq \sum_{i=1}^{m} \P \left( D_i \left( T \right) > t \right) + \sum_{i=m+1}^{\left|T\right|} \P \left( D_i \left( T \right) > t \right) + \sum_{i=\left|T\right|+1}^{\infty} \P \left( D_i \left( T \right) > t \right).
\]
Then using Lemma~\ref{lem:that_we_need} we get the upper bound required for~\eqref{eq:max_tail_exact}: the first sum gives the right hand side of~\eqref{eq:max_tail_exact}, while the other two sums are of smaller order. 
For the lower bound we first have that
\begin{equation}\label{eq:lower}
 \P \left( D_{\max} \left( T \right) > t \right) \geq \sum_{i=1}^{m} \P \left( D_i \left( T \right) > t \right) - \sum_{i=1}^{m} \sum_{j=i+1}^{m} \P \left( D_i \left( T \right) > t, D_j \left( T \right) > t \right).
\end{equation}
Lemma~\ref{lem:that_we_need}(a) with $U = \left\{ i, j \right\}$ implies that for any $1 \leq i < j \leq {m}$,
\begin{equation}\label{eq:ij_bound}
\P \left( D_i \left( T \right) > t, D_j \left( T \right) > t \right) \leq \P \left( D_i \left( T \right) + D_j \left( T \right) > 2t \right) \leq C_{i,j} \left( T \right) t^{1 - 2 \left|T\right| + 4 \Delta \left( T \right)} \exp \left( - t^2 \right)
\end{equation}
for some constant $C_{i,j} \left( T \right)$ and all $t$ large enough. The exponent $-t^2$, appearing on the right hand side of~\eqref{eq:ij_bound}, is smaller by a constant factor than the exponent $-t^2 / 4$, appearing in the asymptotic expression for $\P \left( D_i \left( T \right) > t \right)$ (see~\eqref{eq:tail_seed}). Consequently the second sum on the right hand side of~\eqref{eq:lower} is of smaller order than the first sum, and so we have that $\P \left( D_{\max} \left( T \right) > t \right) \geq \left( 1 - o \left( 1 \right) \right) \sum_{i=1}^{m} \P \left( D_i \left( T \right) > t \right)$ as $t \to \infty$. We can conclude using Lemma~\ref{lem:that_we_need}.
\end{proof}

\section{Distinguishing trees using the maximum degree}\label{sec:distinguishing}

In this section we first prove Theorems~\ref{th:diff_size_deg} and~\ref{th:arbstars}, both using Corollary~\ref{cor:max_tail} (see Section~\ref{sec:proofs}). 
Then in Section~\ref{sec:candidate} we describe a particular way of generalizing the notion of maximum degree which we believe should provide a way to prove Conjecture~\ref{conj:main}. At present we are missing a technical result, see Conjecture~\ref{conj:technical} below, and we prove Conjecture~\ref{conj:main} assuming that this holds.

\subsection{Proofs}\label{sec:proofs}

\begin{proof}\textbf{of Theorem~\ref{th:diff_size_deg}}
 We first provide a simple proof of distinguishing two trees of the same size but with different maximum degree, and then show how to extend this argument to the other cases.

\medskip

\textbf{Case 1: $\left|S\right| - \Delta \left( S \right) \neq \left| T \right| - \Delta \left( T \right)$.}  W.l.o.g.\ suppose that $\left|S\right| - \Delta \left( S \right) < \left| T \right| - \Delta \left( T \right)$. Clearly for any $t > 0$ and $n \geq \max \left\{ \left|S\right|, \left|T \right| \right\}$ one has
\begin{align*}
 \TV \left( \PA \left( n, S \right) , \PA \left( n, T \right) \right) &\geq \TV \left( \Delta \left( \PA \left(n, S \right) \right), \Delta \left( \PA \left( n, T \right) \right) \right) \\
 &\geq \P \left( \Delta \left( \PA \left(n, S \right) \right) > t \sqrt{n} \right) - \P \left( \Delta \left( \PA \left(n, T \right) \right) > t \sqrt{n} \right).
\end{align*}
Taking the limit as $n \to \infty$ this implies that
\begin{equation}\label{eq:delta_D_max}
 \delta \left( S, T \right) \geq \sup_{t > 0} \left[ \P \left( D_{\max} \left( S \right) > t \right) - \P \left( D_{\max} \left( T \right) > t \right) \right].
\end{equation}
By Corollary~\ref{cor:max_tail} and the fact that $\left|S\right| - \Delta \left( S \right) < \left| T \right| - \Delta \left( T \right)$ we have that $\P \left( D_{\max} \left( S \right) > t \right) > \P \left( D_{\max} \left( T \right) > t \right)$ for large enough $t$, which concludes the proof in this case.

\medskip

\textbf{Case 2: $\left|S \right| \neq \left|T\right|$.} W.l.o.g.\ suppose that $\left|S\right| < \left| T \right|$. If $\left|S\right| - \Delta \left( S \right) \neq \left| T \right| - \Delta \left( T \right)$ then by Case~1 we have that $\delta \left( S , T \right) > 0$, so we may assume that $\left|S\right| - \Delta \left( S \right) = \left| T \right| - \Delta \left( T \right)$. Just as in the proof of Case~1 we have that 
\begin{equation}\label{eq:delta_lower2}
 \delta \left( S, T \right) \geq \sup_{t > 0} \left[ \P \left( D_{\max} \left( T \right) > t \right) - \P \left( D_{\max} \left( S \right) > t \right) \right].
\end{equation}
Corollary~\ref{cor:max_tail} provides the asymptotic behavior for $\P \left( D_{\max} \left( T \right) > t \right)$ in the form of~\eqref{eq:max_tail_exact}, where $m \geq 1$. 

\medskip

To find an upper bound for $\P \left( D_{\max} \left( S \right) > t \right)$, first notice that $\Delta \left( \PA \left( \left| T \right|, S \right) \right) \leq \Delta \left(T \right)$, with equality holding if and only if all of the $\left| T \right| - \left| S \right|$ vertices of $\PA \left( \left| T \right|, S \right)$ that were added to $S$ connect to the same vertex $i \in \left\{ 1, 2, \dots, \left|S \right| \right\}$ and $d_S \left( i \right) = \Delta \left( S \right)$. Consequently, if $\Delta \left( \PA \left( \left| T \right|, S \right) \right) = \Delta \left( T \right)$, then there is exactly one vertex $j \in \left\{ 1, 2, \dots, \left| T \right| \right\}$ such that $d_{\PA \left( \left| T \right|, S \right)} \left( j \right) = \Delta \left( T \right)$. This, together with Corollary~\ref{cor:max_tail}, shows that on the one hand
\[
 \P \left( D_{\max} \left( S \right) > t \, \middle| \, \Delta \left( \PA \left( \left| T \right|, S \right) \right) < \Delta \left( T \right) \right) = o \left( t^{1-2\left|T\right| + 2 \Delta \left( T \right)} \exp \left( -t^2 / 4 \right)  \right),
\]
as $t \to \infty$, and on the other hand
\[
 \P \left( D_{\max} \left( S \right) > t \, \middle| \, \Delta \left( \PA \left( \left| T \right|, S \right) \right) = \Delta \left( T \right) \right) \leq  \left( 1 + o \left( 1 \right) \right) c \left( \left|T \right|, \Delta \left( T \right) \right) t^{1-2\left|T\right| + 2 \Delta \left( T \right)} \exp \left( -t^2 / 4 \right)
\]
as $t \to \infty$. Consequently we have that 
\[
 \P \left( D_{\max} \left( S \right) > t \right) \leq \left( 1 + o \left( 1 \right) \right) \P \left( \Delta \left( \PA \left( \left| T \right|, S \right) \right) = \Delta \left( T \right) \right) c \left( \left|T \right|, \Delta \left( T \right) \right) t^{1-2\left|T\right| + 2 \Delta \left( T \right)} \exp \left( -t^2 / 4 \right)
\]
as $t \to \infty$, which combined with the tail behavior of $D_{\max} \left( T \right)$ gives that
\begin{multline*}
 \P \left( D_{\max} \left( T \right) > t \right) - \P \left( D_{\max} \left( S \right) > t \right) \\
\geq \left( 1 - o \left( 1 \right) \right) \P \left( \Delta \left( \PA \left( \left| T \right|, S \right) \right) < \Delta \left( T \right) \right) c \left( \left|T \right|, \Delta \left( T \right) \right) t^{1-2\left|T\right| + 2 \Delta \left( T \right)} \exp \left( -t^2 / 4 \right)
\end{multline*}
as $t\to \infty$. To conclude the proof, notice that $\P \left( \Delta \left( \PA \left( \left| T \right|, S \right) \right) < \Delta \left( T \right) \right)$ is at least as great as the probability that vertex $\left|S\right|+1$ connects to a leaf of $S$, which has probability at least $1/\left( 2 \left|S\right| - 2 \right)$.

\medskip

\textbf{Case 3: $|S|=|T|$, different degree profiles.} Let $z \in \{1,\hdots,|T|\}$ be the first index such that $d_S(z) \neq d_T(z)$ and assume w.l.o.g.\ that $d_S(z) < d_T(z)$. First we have that
\[ \P \left( D_{\max} \left( T \right) > t \right) \geq \P \left(\exists i \in [z-1] : D_i \left( T \right) > t \right) + \P \left(D_z \left( T \right) > t \right) - \sum_{i=1}^{z-1} \P \left( D_z \left( T \right) > t , D_i \left( T \right) > t \right)
\]
and 
\[
 \P \left( D_{\max} \left( S \right) > t \right) \leq \P \left(\exists i \in [z-1] : D_i \left( S \right) > t \right) + \sum_{i=z}^{\infty} \P \left( D_i \left( S \right) > t \right).
\]
Now observe that one can couple the evolution of $\PA(n,T)$ and $\PA(n,S)$ in such a way that the degrees of vertices $1, \hdots, z-1$ stay the same in both trees. Thus one clearly has 
$$\P \left(\exists i \in [z-1] : D_i \left( T \right) > t \right) = \P \left(\exists i \in [z-1] : D_i \left( S \right) > t \right) .$$
Putting the three above displays together one obtains
\begin{multline*}
 \P \left( D_{\max} \left( T \right) > t \right) - \P \left( D_{\max} \left( S \right) > t \right) \\
 \geq \P \left(D_z \left( T \right) > t \right) - \sum_{i=1}^{z-1} \P \left( D_z \left( T \right) > t , D_i \left( T \right) > t \right) - \sum_{i=z}^{\infty} \P \left( D_i \left( S \right) > t \right).
\end{multline*}
Now using Lemma~\ref{lem:that_we_need} one easily gets (for some constant $C>0$) that 
\begin{align*}
& \P \left(D_z \left( T \right) > t \right) \sim c \left( \left| T \right|, d_T(z) \right) t^{1 - 2\left|T\right| + 2d_T(z)} \exp \left( - t^2 / 4 \right) ,\\
& \sum_{i=1}^{z-1} \P \left( D_z \left( T \right) > t , D_i \left( T \right) > t \right) \leq \sum_{i=1}^{z-1} \P \left( D_z \left( T \right) + D_i \left( T \right) > 2t \right) \\
&\qquad \qquad \qquad  \leq \sum_{i=1}^{z-1} \left(1 + o \left( 1 \right) \right) c \left( \left| T \right|, d_T(z) + d_T(i) \right) (2t)^{1 - 2\left|T\right| + 2(d_T(z) + d_T(i))} \exp \left( - t^2 \right) ,\\
& \sum_{i=z}^{\infty} \P \left( D_i \left( S \right) > t \right) \leq C t^{1-2|T| + 2d_S(z)} \exp \left( - t^2 / 4 \right) . 
\end{align*}
In particular, since $d_S(z) < d_T(z)$ and $t^{\alpha} \exp(-t^2) = o(\exp(-t^2/4))$ for any $\alpha$, this shows that 
\[
 \P \left( D_{\max} \left( T \right) > t \right) - \P \left( D_{\max} \left( S \right) > t \right) \geq \left( 1 - o \left( 1 \right) \right) c \left( \left| T \right|, d_T(z) \right) t^{1 - 2\left|T\right| + 2d_T(z)} \exp \left( - t^2 / 4 \right),
\]
which, together with~\eqref{eq:delta_lower2}, concludes the proof.
\end{proof}

\begin{proof} \textbf{of Theorem~\ref{th:arbstars}} 
 As before we have that 
\begin{align}
 \delta \left( S_k, T \right) &\geq \sup_{t \geq 0} \left[ \P \left( D_{\max} \left( S_k \right) > t \right) - \P \left( D_{\max} \left( T \right) > t \right) \right] \notag\\
 &\geq \P \left( D_{\max} \left( S_k \right) > \sqrt{k}/2 \right) - \P \left( D_{\max} \left( T \right) > \sqrt{k}/2 \right). \label{eq:arbstars_proof}
\end{align}
By Corollary~\ref{cor:max_tail}, we know that the second term in~\eqref{eq:arbstars_proof} goes to zero as $k \to \infty$ for any fixed $T$. 
 We can lower bound the first term in~\eqref{eq:arbstars_proof} by $\P \left( D_{1} \left( S_k \right) > \sqrt{k}/2 \right) = 1 - \P \left( D_{1} \left( S_k \right) \leq \sqrt{k}/2 \right)$. From~\eqref{eq:mom_Janson_gen} we have that the first two moments of $D_1 \left( S_k \right)$ are $\E D_1 \left( S_k \right) = \Gamma \left( k \right) / \Gamma \left( k - 1/2 \right)$ and $\E  D_1 \left( S_k \right)^2 = \Gamma \left( k + 1 \right) / \Gamma \left( k \right) = k$. From standard facts about the $\Gamma$ function and Stirling series one has that $0 \leq \E D_1 \left( S_k \right) - \sqrt{k-1} \leq \left( 6 \sqrt{k-1} \right)^{-1}$ and then also
\[
 \Var \left( D_1 \left( S_k \right) \right) = \E  D_1 \left( S_k \right)^2 - \left( \E  D_1 \left( S_k \right) \right)^2 \leq k - \left( k - 1 \right) = 1.
\]
 Therefore Chebyshev's inequality implies that $\lim_{k \to \infty} \P \left( D_{1} \left( S_k \right) \leq \sqrt{k}/2 \right) = 0$.
\end{proof}

\subsection{Towards a proof of Conjecture~\ref{conj:main}}\label{sec:candidate}

Our proof of Theorem~\ref{th:diff_size_deg} above relied on the precise asymptotic tail behavior of $D_{\max} \left( T \right)$, as described in Corollary~\ref{cor:max_tail}. 
%It is possible that this approach can be pushed further---by computing higher order terms of the tail---in order to be able to distinguish trees that are of the same size but which have different degree counts. 
In order to distinguish two trees with the same degree profile (such as the pair of trees in Figure~\ref{fig:example}), it is necessary to incorporate information about the graph structure. Indeed, if $S$ and $T$ have the same degree profiles, then it is possible to couple $\PA \left( n, S \right)$ and $\PA \left( n, T \right)$ such that they have the same degree profiles for every $n$.

Thus a possible way to prove Conjecture~\ref{conj:main} is to generalize the notion of maximum degree in a way that incorporates information about the graph structure, and then use similar arguments as in the proofs above. A candidate is the following.

\begin{definition}\label{def:U-max_degree}
 Given a tree $U$, define the \emph{$U$-maximum degree} of a tree $T$, denoted by $\Delta_U \left( T \right)$, as
\[
 \Delta_U \left( T \right) = \max_{\varphi} \sum_{u \in V\left( U \right)} d_T \left( \varphi \left( u \right) \right),
\]
where $V\left( U \right)$ denotes the vertex set of $U$, and the maximum is taken over all injective graph homomorphisms from $U$ to $T$. That is, $\varphi$ ranges over all injective maps from $V\left( U \right)$ to $V \left( T \right)$ such that $\left\{ u, v \right\} \in E \left( U \right)$ implies that $\left\{ \varphi \left( u \right), \varphi \left( v \right) \right\} \in E \left( T \right)$, where $E \left( U \right)$ denotes the edge set of $U$, and $E \left( T \right)$ is defined similarly.
\end{definition}

When $U$ is a single vertex, then $\Delta_U \equiv \Delta$, so this indeed generalizes the notion of maximum degree. We conjecture the following.

\begin{conjecture}\label{conj:technical}
 Suppose $S$ and $T$ are two non-isomorphic trees of the same size. Then
\[
 \limsup_{n \to \infty} \P \left( \Delta_T \left( \PA \left( n, S \right) \right) > t \sqrt{n} \right) = o \left(   t^{2\left|T\right| - 3} \exp \left( - t^2 / 4 \right) \right)
\]
as $t\to \infty$.
\end{conjecture}

If this conjecture were true, then Conjecture~\ref{conj:main} also follows, as we now show.

\medskip

\begin{proof}\textbf{of Conjecture~\ref{conj:main} assuming Conjecture~\ref{conj:technical} holds} Assume $\left| S \right| = \left| T \right|$; if $\left| S \right| \neq \left| T \right|$ we already know from Theorem~\ref{th:diff_size_deg} that $\delta \left( S, T \right) > 0$. 
As in the proof of Theorem~\ref{th:diff_size_deg}, for any $t > 0$ and $n \geq \max \left\{ \left|S\right|, \left|T \right| \right\}$ we have that 
\begin{align*}
 \TV \left( \PA \left( n, S \right) , \PA \left( n, T \right) \right) &\geq \TV \left( \Delta_T \left( \PA \left(n, S \right) \right), \Delta_T \left( \PA \left( n, T \right) \right) \right) \\
 &\geq \P \left( \Delta_T \left( \PA \left(n, T \right) \right) > t \sqrt{n} \right) - \P \left( \Delta_T \left( \PA \left(n, S \right) \right) > t \sqrt{n} \right),
\end{align*}
 and consequently
\begin{equation}\label{eq:delta_lower}
 \delta \left( S, T \right) \geq \sup_{t > 0} \left\{ \liminf_{n \to \infty} \P \left( \Delta_T \left( \PA \left(n, T \right) \right) > t \sqrt{n} \right) - \limsup_{n \to \infty} \P \left( \Delta_T \left( \PA \left(n, S \right) \right) > t \sqrt{n} \right) \right\}.
\end{equation}
Since $\varphi \left( i \right) = i$ for $1 \leq i \leq \left| T \right|$ is an injective graph homomorphism from $T$ to $\PA \left( n, T \right)$, we have that 
\[
 \liminf_{n \to \infty} \P \left( \Delta_T \left( \PA \left(n, T \right) \right) > t \sqrt{n} \right) \geq \liminf_{n \to \infty} \P \left( \sum_{i=1}^{\left|T\right|} d_{\PA \left( n, T \right)} \left( i \right) > t \sqrt{n} \right) = \P \left( \sum_{i=1}^{\left|T\right|} D_i \left( T \right) > t \right).
\]
By Lemma~\ref{lem:that_we_need} we know that 
\[
 \P \left( \sum_{i=1}^{\left|T\right|} D_i \left( T \right) > t \right) \sim c \left( \left| T\right|, 2\left|T\right| - 2 \right) t^{2 \left| T \right| - 3} \exp \left( - t^2 / 4 \right)
\]
as $t \to \infty$, which together with~\eqref{eq:delta_lower} and Conjecture~\ref{conj:technical} shows that $\delta \left( S, T \right) > 0$.
\end{proof}

\section{The weak limit of $\PA(n,T)$}\label{sec:weaklimit}

In this section we prove Theorem~\ref{th:weaklimit}. 
For two graphs $G$ and $H$ we write $G=H$ if $G$ and $H$ are isomorphic, and we use the same notation for rooted graphs. Recalling the definition of the Benjamini-Schramm limit (see [Definition 2.1., \cite{BBCS12}]), we want to prove that
$$\lim_{n \to \infty} \P\bigg(B_r(\PA(n,T), k_n(T)) = (H,y) \bigg) = \P\bigg(B_r(\cT, (0)) = (H,y)\bigg) ,$$
where $B_r(G,v)$ is the rooted ball of radius $r$ around vertex $v$ in the graph $G$, $k_n(T)$ is a uniformly random vertex in $\PA(n,T)$, $(H,y)$ is a finite rooted tree and $(\cT, (0))$ is the P{\'o}lya-point graph (with $m=1$). 

\medskip

We construct a forest $F$ based on $T$ as follows. To each vertex $v$ in $T$ we associate $d_T(v)$ isolated nodes with self loops, that is $F$ consists of $2 (|T| -1)$ isolated vertices with self loops. Our convention here is that a node with $k$ regular edges and one self loop has degree $k+1$. The graph evolution process $\PA(n,F)$ for forests is defined in the same way as for trees, and we couple the processes $\PA(n,T)$ and $\PA(n+|T|-2,F)$ in the natural way: when an edge is added to vertex $v$ of $T$ in $\PA(n,T)$ then an edge is also added to one of the $d_T(v)$ corresponding vertices of $F$ in $\PA(n+|T|-2,F)$, and furthermore newly added vertices are always coupled. We first observe that, clearly, the weak limit of $\PA(n+|T|-2,F)$ is the P{\'o}lya-point graph, that is
$$\lim_{n \to \infty} \P\bigg(B_r(\PA(n+|T|-2,F), k_n(F)) = (H,y) \bigg) = \P\bigg(B_r(\cT, (0)) = (H,y)\bigg) ,$$
where $k_n(F)$ is a uniformly random vertex in $\PA(n+|T|-2,F)$. We couple $k_n(F)$ and $k_n(T)$ in the natural way, that is if $k_n(F)$ is the $t^{th}$ newly created vertex in $\PA(n+|T|-2, F)$ then $k_n(T)$ is the $t^{th}$ newly created vertex in $\PA(n,T)$.
%(given that we study $r$-neighborhoods)
%, that is if $k_n(T)$ is such that $B_r(\PA(n,T), k_n(T))$ does not contain one of the original seed vertices from $T$ then $k_n(F) = k_n(T)$. 
To conclude the proof it is now sufficient to show that
$$\lim_{n \to \infty} \P\bigg(B_r(\PA(n+|T|-2,F), k_n(F)) \neq B_r(\PA(n,T), k_n(T)) \bigg) = 0 .$$
The following inequalities hold true (with a slight---but clear---abuse of notation when we write $v \in F$) for any $u >0$,
\begin{align*}
& \P\bigg(B_r(\PA(n+|T|-2,F), k_n(F)) \neq B_r(\PA(n,T), k_n(T)) \bigg) \\
& \leq \P\bigg(\exists v \in F \; \text{s.t.} \; v \in B_r(\PA(n+|T|-2,F), k_n(F)) \bigg) \\
& \leq \P\bigg(\exists v \in F, d_{\PA(n+|T|-2,F)}(v) < u\bigg) \\
&\quad + \P\bigg(\exists v \in B_r(\PA(n+|T|-2,F), k_n(F)) \; \text{s.t.} \; d_{\PA(n+|T|-2,F)}(v) \geq u\bigg) .
\end{align*}
It is easy to verify that for any $u > 0$,
$$\lim_{n \to \infty} \P\bigg(\exists v \in F, d_{\PA(n+|T|-2,F)}(v) < u\bigg) = 0 .$$
Furthermore since $B_r(\PA(n+|T|-2,F)$ tends to the P{\'o}lya-point graph we also have
\begin{multline*}
\lim_{n \to \infty} \P\bigg(\exists v \in B_r(\PA(n+|T|-2,F), k_n(F)) \; \text{s.t.} \; d_{\PA(n+|T|-2,F)}(v) \geq u\bigg) \\
 = \P\bigg(\exists v \in B_r(\cT, (0)) \; \text{s.t.} \; d_{\cT}(v) \geq u\bigg). 
\end{multline*}
By looking at the definition of $(\cT, (0))$ given in \cite{BBCS12} one can easily show that
$$\lim_{u \to \infty} \P\bigg(\exists v \in B_r(\cT, (0)) \; \text{s.t.} \; d_{\cT}(v) \geq u\bigg) = 0 ,$$
which concludes the proof.

\section{Proof of Lemma~\ref{lem:that_we_need}}\label{sec:lemma_proof}

In this section we prove Lemma~\ref{lem:that_we_need}. 
In light of the representation~\eqref{eq:rep} in Section~\ref{sec:previous_arbitrary}, part (a) of Lemma~\ref{lem:that_we_need} follows from a lengthy computation, the result of which we state separately.

\begin{lemma}\label{lem:computation}
 Fix positive integers $a$ and $b$. Let $B$ and $Z$ be independent random variables such that $B \sim \Beta \left( a, b \right)$ and $Z \sim \GGa \left( a + b + 1, 2 \right)$, and let $V = 2 B Z$. Then
\begin{equation}\label{eq:tail_asymp}
 \P \left( V > t \right) \sim c \left( \frac{a+b+2}{2}, a \right) t^{-1+a-b} \exp \left( - t^2 / 4 \right)
\end{equation}
 as $t\to\infty$, where the constant $c$ is as in~\eqref{eq:const}.
\end{lemma}

\begin{proof}
  By definition we have for $t > 0$ that
 \begin{align*}
 \P \left( V > t \right) &= \P\left( 2BZ > t \right) = \int_{t/2}^{\infty} \int_{t/(2z)}^1 \frac{\Gamma(a+b)}{\Gamma(a) \Gamma(b)} x^{a-1} \left(1-x \right)^{b-1} dx \frac{2}{\Gamma\left(\frac{a+b+1}{2}\right)} z^{a+b} e^{-z^2} dz \\
 &= \int_{t/2}^{\infty} \left[ 1 - I_{t/(2z)} (a,b) \right] \frac{2}{\Gamma\left(\frac{a+b+1}{2}\right)} z^{a+b} e^{-z^2} dz,
\end{align*}
where $I_x (a,b) = \frac{\Gamma(a+b)}{\Gamma(a)\Gamma(b)} \int_0^x y^{a-1} (1-y)^{b-1} dy$ is the regularized incomplete Beta function. For positive integers $a$ and $b$, integration by parts and induction gives that
\[
 I_x (a,b) = 1 - \sum_{j=0}^{a-1} \binom{a+b-1}{j} x^j \left( 1- x \right)^{a+b-1-j}.
\]
Plugging this back in to the integral and doing a change of variables $y = 2z$, we get that
\[
 \P( V > t ) = \frac{2^{-\left(a+b\right)}}{\Gamma \left( \frac{a+b+1}{2} \right)} \sum_{j=0}^{a-1} \binom{a+b-1}{j} \int_t^\infty t^j \left( y - t \right)^{a+b-1-j} y \exp \left(-y^2 / 4 \right) dy.
\]
Expanding $\left( y - t \right)^{a+b-1-j}$ we arrive at the alternating sum formula
\begin{equation}\label{eq:tail_alt_sum}
 \P \left( V > t \right) = \frac{2^{-\left(a+b\right)}}{\Gamma \left( \frac{a+b+1}{2} \right)} \sum_{j=0}^{a-1} \sum_{k=0}^{a+b-1-j} \binom{a+b-1}{j} \binom{a+b-1-j}{k} \left( - 1 \right)^{a+b-1-j-k} t^{a+b-1-k} A_{k+1},
\end{equation}
where for $m \geq 0$ let 
\[
 A_m := \int_t^\infty y^m \exp \left( - y^2 / 4 \right) dy.
\]
Thus in order to show~\eqref{eq:tail_asymp} it is enough to show that for every $j$ such that $0 \leq j \leq a - 1$ we have
\begin{equation}\label{eq:to_show}
 \sum_{k=0}^{a+b-1-j} \binom{a+b-1-j}{k} \left( - 1 \right)^{a+b-1-j-k} t^{a+b-1-k} A_{k+1} \sim  \frac{2^{a+b-j} (a+b-1-j)!}{t^{a+b-1-2j}} \exp \left( - t^2 / 4 \right).
\end{equation}
To do this, we need to evaluate the integrals $\left\{ A_m \right\}_{m \geq 0}$. Recall that the complementary error function is defined as $\erfc \left( z \right) = 1 - \erf \left( z \right) = \left( 2 / \sqrt{\pi} \right) \int_z^{\infty} \exp \left( - u^2 \right) du$, and thus $A_0 = \sqrt{\pi} \erfc\left( t/2 \right)$; also $A_1 = 2 \exp \left( - t^2 / 4 \right)$. Integration by parts gives that for $m \geq 2$ we have $A_m = 2 t^{m-1} \exp \left( - t^2 / 4 \right) + 2 \left( m - 1 \right) A_{m-2}$. Iterating this, and using the values for $A_0$ and $A_1$, gives us that for $m$ odd we have
\begin{equation}\label{eq:A_m_odd}
 A_m = 2 t^{m-1} \exp \left( - t^2 / 4 \right)  \sum_{\ell = 0}^{\frac{m-1}{2}} \frac{\left( m - 1 \right)!!}{\left( m - 2\ell - 1 \right)!!} \left( \frac{2}{t^2} \right)^{\ell},
\end{equation}
and for $m$ even we have
\begin{equation}\label{eq:A_m_even}
 A_m = 2 t^{m-1} \exp \left( - t^2 / 4 \right) \sum_{\ell = 0}^{\frac{m}{2}-1} \frac{\left( m - 1 \right)!!}{\left( m - 2\ell - 1 \right)!!} \left( \frac{2}{t^2} \right)^{\ell} + 2^{\frac{m}{2}} \times \left( m - 1 \right)!! \times \sqrt{\pi} \erfc \left( t/2 \right) .
\end{equation}
In the following we fix $j$ such that $0 \leq j \leq a - 1$ and $a+b-1-j$ is odd---showing~\eqref{eq:to_show} when $a+b-1-j$ is even can be done in the same way. In order to abbreviate notation we let $r = (a+b-2-j)/2$. Plugging in the formulas~\eqref{eq:A_m_odd} and~\eqref{eq:A_m_even} into the left hand side of~\eqref{eq:to_show} we get that 
\begin{align}
 &\sum_{k=0}^{a+b-1-j} \binom{a+b-1-j}{k} \left( - 1 \right)^{a+b-1-j-k} t^{a+b-1-k} A_{k+1} = \sum_{k=0}^{2r+1} \binom{2r+1}{k} \left( - 1 \right)^{2r+1-k} t^{2r+1+j-k} A_{k+1} \notag\\
&\qquad\qquad = - \sum_{\ell = 0}^{r} \binom{2r+1}{2\ell} t^{2r+1+j-2\ell} A_{2\ell +1} + \sum_{\ell=0}^{r} \binom{2r+1}{2\ell+1} t^{2r+1+j-(2\ell+1)} A_{2\ell+2} \notag\\
&\qquad\qquad =  - \sum_{\ell = 0}^{r} \binom{2r+1}{2\ell} t^{2r+1+j-2\ell} 2 \exp\left( - t^2 / 4 \right) \sum_{u=0}^{\ell} 2^u \frac{\left(2\ell\right)!!}{\left(2\ell-2u\right)!!} t^{2\ell-2u} \notag\\
&\qquad\qquad \quad + \sum_{\ell=0}^{r} \binom{2r+1}{2\ell+1} t^{2r+1+j-(2\ell+1)} 2 \exp\left( - t^2 / 4 \right) \sum_{u=0}^\ell 2^u \frac{\left(2\ell+1\right)!!}{\left(2\ell+1-2u\right)!!} t^{2\ell+1-2u} \notag\\
&\qquad\qquad \quad + \sum_{\ell=0}^{r} \binom{2r+1}{2\ell+1} t^{2r+1+j-(2\ell+1)} 2^{\ell+1} \left( 2\ell+1\right)!! \sqrt{\pi} \erfc\left(t/2\right) \notag\\
&\qquad\qquad = 2 \exp \left( - t^2/4\right) \sum_{u=0}^{r} t^{2r+1+j-2u} 2^u \sum_{k=2u}^{2r+1} \binom{2r+1}{k}\left(-1\right)^{k+1} \frac{k!!}{\left(k-2u\right)!!} \label{eq:sum1}\\
&\qquad\qquad \quad + \sqrt{\pi} \erfc \left(t/2\right) \sum_{\ell=0}^{r} \binom{2r+1}{2\ell+1} t^{2r+1+j-(2\ell+1)} 2^{\ell+1} \left( 2\ell+1\right)!!. \label{eq:sum2}
\end{align}
An important fact that we will use is that for every polynomial $P$ with degree less than $n$ we have
\begin{equation}\label{eq:alt_sum_zero}
 \sum_{k=0}^n \binom{n}{k} \left( - 1 \right)^k P(k) = 0.
\end{equation}
Consequently, applying this to the polynomial $P(k) = k \left( k - 2 \right) \cdots \left( k - 2 \left( u - 1 \right) \right)$ we get that
\begin{align}
 &\sum_{k=2u}^{2r+1} \binom{2r+1}{k} \left(-1\right)^{k+1} k\left(k-2\right) \cdots \left( k - 2\left(u-1\right) \right) \notag\\
&\qquad \qquad \qquad \qquad= \sum_{k=0}^{2u-1} \binom{2r+1}{k} \left(-1\right)^k  k\left(k-2\right) \cdots \left( k - 2\left(u-1\right) \right)\notag \\
&\qquad \qquad \qquad \qquad= - \sum_{\ell=0}^{u-1} \binom{2r+1}{2\ell+1} \left(2\ell+1\right)\left(2\ell-1\right) \cdots \left( 2\ell+1 - 2 \left(u-1\right) \right) \notag \\
&\qquad \qquad \qquad \qquad= - \sum_{\ell=0}^{u-1} \binom{2r+1}{2\ell+1} \left(2\ell+1\right)!! \left( 2 \left( u - 1 - \ell \right) - 1 \right)!! \left( - 1 \right)^{u-1-\ell}. \label{eq:coeff_2a}
\end{align}
Thus we see that in the sum~\eqref{eq:sum1} the cofficient of the term involving $t^{2r+1+j}$ is zero, while the coefficient of the term involving $t^{2r+1+j-2u}$ for $1 \leq u \leq r$ is $2^{u+1} \exp \left( - t^2 / 4 \right)$ times the expression in~\eqref{eq:coeff_2a}. These are cancelled by terms coming from the sum in~\eqref{eq:sum2} as we will see shortly; to see this we need the asymptotic expansion of $\erfc$ to high enough order. In particular we have (see~\cite[equations~7.1.13 and~7.1.24]{abramowitz1964handbook}) that
\begin{equation}\label{eq:erfc_expansion}
  \sqrt{\pi} \erfc \left( t / 2 \right) = 2 \exp \left( - t^2 / 4 \right) \sum_{n=0}^{2r} \left( - 1 \right)^n 2^n \left( 2n - 1 \right)!! t^{-2n-1} + R \left( t \right),
\end{equation}
where the approximation error $R \left( t \right)$ satisfies
\[
 \left| R \left( t \right) \right| \leq 2^{2r+2} \left( 4r+1 \right)!!  t^{-(4r+3)} \exp \left( - t^2 / 4 \right).
\]
Plugging~\eqref{eq:erfc_expansion} back into~\eqref{eq:sum2}, we first see that the error term satisfies
\begin{equation}\label{eq:error}
 \left| R \left( t \right) \right| \sum_{\ell=0}^{r} \binom{2r+1}{2\ell+1} t^{2r+1+j-(2\ell+1)} 2^{\ell+1} \left( 2\ell+1\right)!! = O \left( t^{2j-1-(a+b)} \exp \left( - t^2 / 4 \right) \right)
\end{equation}
as $t \to \infty$. The main term of~\eqref{eq:sum2} becomes the sum
\[
 2 \exp \left( - t^2 / 4 \right) \sum_{\ell=0}^{r} \sum_{n=0}^{2r} \binom{2r+1}{2\ell+1} 2^{\ell+n+1} \left( 2 \ell + 1 \right)!! \left( 2n - 1 \right)!! \left( - 1 \right)^n t^{2r+1+j-2 \left( \ell + n + 1 \right)}.
\]
For $u$ such that $1 \leq u \leq r$, the coefficient of the term involving $t^{2r+1+j-2u}$ is $2^{u+1} \exp \left(-t^2/4\right)$ times
\[
 \sum_{\ell=0}^{u-1} \binom{2r+1}{2\ell+1} \left(2\ell+1\right)!! \left( 2 \left( u - 1 - \ell \right) - 1 \right)!! \left( - 1 \right)^{u-1-\ell},
\]
which cancels out the coefficient of the same term coming from the other sum~\eqref{eq:sum1}, see~\eqref{eq:coeff_2a}. 
For $u$ such that $r < u \leq 2r$, the coefficient of the term involving $t^{2r+1+j-2u}$ is $2^{u+1} \exp \left(-t^2/4\right)$ times 
\begin{multline*}
 \sum_{\ell=0}^{r} \binom{2r+1}{2\ell+1} \left(2\ell+1\right)!! \left( 2 \left( u - 1 - \ell \right) - 1 \right)!! \left( - 1 \right)^{u-1-\ell}\\
\begin{aligned}
&=  \sum_{\ell=0}^{r} \binom{2r+1}{2\ell+1} \left( 2 \ell + 1 \right) \left( 2 \ell -1 \right) \dots \left( \left( 2\ell+1 \right) - 2 \left( u - 1 \right) \right) \\
&= - \sum_{k=0}^{2r+1} \binom{2r+1}{k} \left(-1\right)^k k \left(k - 2 \right) \dots \left( k - 2 \left( u-1 \right) \right) = 0,
\end{aligned}
\end{multline*}
where we again used~\eqref{eq:alt_sum_zero}, together with the fact that $u \leq 2r$. Finally, the coefficient of the term involving $t^{2j+1-(a+b)}$ is $2^{2r+2} \exp \left(-t^2/4\right)$ times
\begin{multline*}
 \sum_{\ell=0}^{r} \binom{2r+1}{2\ell+1} \left(2\ell+1\right)!! \left( 2 \left( 2r - \ell \right) - 1 \right)!! \left( - 1 \right)^{2r-\ell} =  - \sum_{k=0}^{2r+1} \binom{2r+1}{k} \left(-1\right)^k k \left(k - 2 \right) \dots \left( k - 4r \right) \\
= - \sum_{k=0}^{2r+1} \binom{2r+1}{k} \left(-1\right)^k  k^{2r+1} = - \left( - 1\right)^{2r+1} \left(2r+1\right)! = \left(2r+1\right)!,
\end{multline*}
where we used~\eqref{eq:alt_sum_zero} in the second equality. 
Since all other terms are of lower order (see~\eqref{eq:error}), this concludes the proof.
\end{proof}

\begin{proof} \textbf{of Lemma~\ref{lem:that_we_need}} 
(a) If $U \neq T$, then $d = \sum_{i \in U} d_T \left( i \right) \in \left\{ 1, \dots, 2 \left| T \right| - 3 \right\}$. 
Similarly to the third paragraph in Section~\ref{sec:previous_arbitrary}, we can view the evolution of $\sum_{i\in U} d_{\PA \left(n,T\right)} \left( i \right)$ in the following way. 
When adding a new vertex, first decide whether it attaches to one of the initial $\left|T\right|$ vertices (with probability $\sum_{i=1}^{\left|T\right|} d_{\PA \left(n,T\right)} \left( i \right) / \left( 2n - 2 \right)$) or not (with the remaining probability); if it does, then independently pick one of them to attach to with probability proportional to their degree---a vertex in $U$ is chosen with probability $\sum_{i \in U} d_{\PA \left( n, T \right)} \left( i \right) / \sum_{i=1}^{\left|T\right|} d_{\PA \left(n,T\right)} \left( i \right)$. 
This implies the following representation: $\sum_{i \in U} D_i \left( T \right) \stackrel{d}{=} 2BZ$, where $B$ and $Z$ are independent, $B \sim \Beta \left( d, 2\left|T\right| - 2 - d \right)$, and $Z \sim \GGa \left( 2\left|T\right|-1, 2 \right)$. 
This also follows directly from the representation~\eqref{eq:rep}. 
    Thus~\eqref{eq:tail_seed} is a direct consequence of Lemma~\ref{lem:computation}.

If $U = T$, then $\sum_{i \in U} D_i \left( T \right) \stackrel{d}{=} 2Z$ where $Z \sim \GGa \left( 2 \left| T \right| - 1, 2 \right)$ (see Section~\ref{sec:previous_arbitrary}), and then~\eqref{eq:tail_seed} follows from a calculation that is contained in the proof of Lemma~\ref{lem:computation}.

(b) To show~\eqref{eq:tail_late} we use the results of~\cite{pekoz2013degree} as described in Section~\ref{sec:previous_edge}. In addition we use the following tail bound of~\cite[Lemma~2.6]{pekoz2013degree}, which says that for $x > 0$ and $s \geq 1$ we have $\int_x^{\infty} \kappa_s \left( y \right) dy \leq \frac{s}{x} \kappa_s \left( x \right)$. Consequently, for any $i > \left| T \right|$ we have the following tail bound:
\begin{align*}
 \P \left( D_i \left( T \right) > t \right) &= \P \left( W_{i-1} > \sqrt{\frac{i-1}{2}} t \right) = \int_{\sqrt{\frac{i-1}{2}} t}^{\infty} \kappa_{i-1} \left( y \right) dy \\
 &\leq \frac{\sqrt{2i-2}}{t} \kappa_{i-1} \left( \sqrt{\frac{i-1}{2}} t \right) = \frac{2}{\sqrt{\pi} t} \exp \left( - t^2 / 4 \right) \left( i-2 \right)! U \left( i-2, \frac{1}{2}, \frac{t^2}{4} \right).
\end{align*}
The following integral representation is useful for us~\cite[eq.~13.2.5]{abramowitz1964handbook}:
\[
 \Gamma \left( a \right) U \left( a, b, z \right) = \int_0^\infty e^{-zw} w^{a-1} \left( 1 + w \right)^{b-a-1} dw.
\]
Consequently, we have
\begin{align*}
 \sum_{i=3}^{\infty} \left( i - 2 \right)! U \left( i-2, \frac{1}{2}, \frac{t^2}{4} \right) &= \sum_{i=3}^\infty \left( i - 2 \right) \int_0^\infty e^{-\frac{t^2}{4} w} \frac{1}{w\sqrt{1+w}} \left( \frac{w}{1+w} \right)^{i-2} dw \\
 &= \int_0^\infty e^{-\frac{t^2}{4} w} \frac{1}{w \sqrt{1+w}} \sum_{i=3}^\infty \left( i - 2 \right) \left( \frac{w}{1+w} \right)^{i-2} dw \\
 &= \int_0^\infty e^{-\frac{t^2}{4} w}\frac{1}{w \sqrt{1+w}} w \left( 1 + w \right) dw \\
 &\leq \int_0^\infty e^{-\frac{t^2}{4} w} \left( 1 + w \right) dw = \frac{4}{t^2} + \frac{16}{t^4},
\end{align*}
which shows~\eqref{eq:tail_late} for $L=3$. Similarly, for $L \geq 4$ we have
\begin{align*}
 \sum_{i=L}^{\infty} \left( i - 2 \right)! U \left( i-2, \frac{1}{2}, \frac{t^2}{4} \right) &= \int_0^\infty e^{-\frac{t^2}{4} w} \frac{1}{w \sqrt{1+w}} \sum_{i=L}^\infty \left( i - 2 \right) \left( \frac{w}{1+w} \right)^{i-2} dw \\
 &= \int_0^\infty e^{-\frac{t^2}{4} w} \frac{1}{w \sqrt{1+w}} \frac{\left( L - 2 \right) \left( \frac{w}{1+w} \right)^{L-2} + \left( 3- L \right) \left( \frac{w}{1+w} \right)^{L-1}}{1/ \left( 1 + w \right)^2} dw \\
 &\leq \int_0^\infty e^{-\frac{t^2}{4} w} \left( L - 2 \right) \left( \frac{w}{1+w} \right)^{L-3} \sqrt{1+w} dw \\
 &\leq \int_0^\infty e^{-\frac{t^2}{4} w} \left( L - 2 \right) w^{L-3} dw = \frac{4^{L-2} \times \left( L -2 \right)!}{t^{2L-4}},
\end{align*}
where the first inequality follows from dropping the nonpositive term $\left( 3- L \right) \left( \frac{w}{1+w} \right)^{L-1}$, and the second one follows because $L \geq 4$. This shows~\eqref{eq:tail_late} for $L \geq 4$ and thus concludes the proof.
\end{proof}

\section{Open problems} \label{sec:op}
\begin{enumerate}
\item The main open question at this stage is whether Conjecture~\ref{conj:main} holds true. 
Our results leave open the case 
$\vec{d}(S) = \vec{d}(T)$.
%when $\left|S\right| = \left|T \right|$, $\Delta \left( S \right) = %\Delta \left( T \right)$, and $m_S = m_T$. 
%In Section~\ref{sec:candidate} we noted that our current approach might be pushed further in order to distinguish between trees with different degree counts, and this might be interesting in itself. 
We described a way to approach Conjecture~\ref{conj:main} in general, and showed that it would follow from a technical result which we stated as Conjecture~\ref{conj:technical}.
\item This paper is essentially about the {\em testing} version of the problem. Can anything be said about the {\em estimation} version? Perhaps a first step would be to understand the multiple hypothesis testing problem where one is interested in testing whether the seed belongs to the family of trees $\cT_1$ or to the family $\cT_2$.
\item Starting from two seeds $S$ and $T$ with different spectrum, is it always possible to distinguish (with non-trivial probability) between $\PA(n,S)$ and $\PA(n,T)$ with spectral techniques? More generally, it would be interesting to understand what properties are invariant under modifications of the seed.
\item Is it possible to give a combinatorial description of the (pseudo)metric $\delta$?
\item Under what conditions on two tree sequences $(T_k)$, $(R_k)$ do we have $\lim_{k \to \infty} \delta(T_k, R_k) = 1$? In Theorem \ref{th:arbstars} we showed that a sufficient condition is to have $T_k = T$ and $R_k = S_k$. This can easily be extended to the condition that $\Delta(T_k)$ remains bounded while $\Delta(R_k)$ tends to infinity. 
If $T_k$ and $R_k$ are independent (uniformly) random trees on $k$ vertices, do we have $\lim_{k \to \infty} \E \delta(T_k, R_k) = 1$?
\item What can be said about the general preferential attachment model, when multiple edges or vertices are added at each step?
\item A simple variant on the model studied in this paper is to consider probabilities of connection proportional to the degree of the vertex raised to some power $\alpha$. 
For $\alpha = 1$ we conjectured and in some cases showed that different seeds are distinguishable. On the contrary, it seems reasonable to expect that for $\alpha = 0$ (uniform attachment) all seeds are indistinguishable from each other asymptotically. What about for $\alpha \in (0,1)$?
\end{enumerate}

\subsection*{Acknowledgements}
The first author thanks Nati Linial for initial discussions on this problem. 
We also thank Remco van der Hofstad and Nathan Ross for helpful discussions and valuable pointers to the literature. 
The research described here was carried out at the Simons Institute for the Theory of Computing. We are grateful to the Simons Institute for offering us such a wonderful research environment. 
This work is supported by NSF grants DMS 1106999 and CCF 1320105 (E.M.), and by DOD ONR grant N000141110140 (E.M., M.Z.R.).

\bibliographystyle{plainnat}
\bibliography{bib}
\end{document}

%% file: Commands.tex
\newtheorem{theorem}{Theorem}

\newtheorem{lemma}{Lemma}
\newtheorem{corollary}{Corollary}

\newtheorem{definition}{Definition}
\newtheorem{conjecture}{Conjecture}

\renewcommand{\phi}{\varphi}

\renewcommand{\P}{\mathbb{P}}
\newcommand{\E}{\mathbb{E}}
\newcommand{\N}{\mathbb{N}}

\newcommand{\cT}{\mathcal{T}}

\def\ds1{\mathds{1}}
\renewcommand{\epsilon}{\varepsilon}

\newcommand\Var{{\dsV\text{ar}}\,}
\newcommand\dsV{\mathbb{V}}

\newlength{\minipagewidth}
\newcommand{\beq}{\begin{equation}}
\newcommand{\eeq}{\end{equation}}

\newcommand{\beqa}{\begin{eqnarray}}
\newcommand{\eeqa}{\end{eqnarray}}

\newcommand{\beqan}{\begin{eqnarray*}}
\newcommand{\eeqan}{\end{eqnarray*}}

\def\ba#1\ea{\begin{align*}#1\end{align*}} %\ba = \begin{algin*}, \ea = \end{align*}
\def\banum#1\eanum{\begin{align}#1\end{align}} %\banum = \begin{algin}, \eanum